\documentclass[11pt]{amsart}

\usepackage[a4paper,margin=28mm]{geometry}
\usepackage{amsmath,amssymb,amsthm,mathtools}
\mathtoolsset{showonlyrefs}
\numberwithin{equation}{section}

\usepackage[numbers,sort&compress]{natbib}
\usepackage{doi}

\usepackage{xcolor}
\usepackage{hyperref}
\usepackage{enumerate}

\hypersetup{
  colorlinks=true,
  linkcolor=blue!45!black,
  citecolor=blue!45!black,
  urlcolor=blue!45!black,
  pdftitle={The frame set of the first Hermite function},
  pdfauthor={Markus Faulhuber and Philipp Petersen}
}

\newtheorem{theorem}{Theorem}[section]
\newtheorem{proposition}[theorem]{Proposition}
\newtheorem{lemma}[theorem]{Lemma}
\newtheorem{corollary}[theorem]{Corollary}
\newtheorem{conjecture}[theorem]{Conjecture}
\theoremstyle{definition}
\newtheorem{definition}[theorem]{Definition}
\theoremstyle{remark}
\newtheorem{remark}[theorem]{Remark}

\newcommand{\R}{\mathbb R}
\newcommand{\C}{\mathbb C}
\newcommand{\Z}{\mathbb Z}
\newcommand{\N}{\mathbb N}
\newcommand{\Q}{\mathbb Q}
\newcommand{\ii}{\mathrm i}
\newcommand{\e}{\mathrm e}
\newcommand{\cG}{\mathcal G}
\newcommand{\cF}{\mathcal F}
\newcommand{\cW}{\mathcal W}
\newcommand{\cZ}{\mathcal Z}
\newcommand{\Hullw}{\operatorname{Hull}_{w}}
\newcommand{\Wr}{\operatorname{Wr}}
\newcommand{\ord}{\operatorname{ord}}

\title{The frame set of the first Hermite function}

\author[M. Faulhuber]{Markus Faulhuber}
\address{
    Faculty of Mathematics,
    University of Vienna,\newline
    Oskar-Morgenstern-Platz 1,
    1090 Vienna, Austria
}
\email{markus.faulhuber@univie.ac.at}

\author[P. Petersen]{Philipp Petersen}
\address{
    Faculty of Mathematics,
    University of Vienna,\newline
    Kolingasse 14-16,
    1090 Vienna, Austria
}
\email{philipp.petersen@univie.ac.at}

\date{}

\thanks{The research of M.~F. was supported by the Austrian Science Fund (FWF) [\href{https://doi.org/10.55776/PAT5102224}{10.55776/PAT5102224}]}

\begin{document}

\begin{abstract}
    We determine the frame set of the first Hermite function, thereby proving a conjecture of Lyubarskii and Nes. A characterization of semi-regular Gabor frames due to Gr\"ochenig, Romero, and St\"ockler reduces the problem to a uniqueness question for entire functions in a Gaussian shift-invariant space. We address this question by forming Wronskians of finitely many translates of a Gaussian shift-invariant function. Their common critical points become zeros of increasing multiplicity, while automorphy shows that the Wronskians remain within a Gaussian shift-invariant class. A sharp zero-density theorem then rules out all possible failures of the frame property except at the known rational obstructions. This yields the complete rectangular frame set of the first Hermite function and introduces Wronskian amplification as a new method in Gabor frame theory.
\end{abstract}

\maketitle

\vspace{-\baselineskip}
\section{\texorpdfstring{Introduction and main result}{Introduction and main result}}\label{sec:intro}

For $x,\omega\in\R$ and $g\in L^2(\R)$, its time-shifts and frequency-shifts (modulations) are
\begin{equation}
    T_x g(t) = g(t-x),
    \qquad
    M_\omega g(t) = \e^{2\pi\ii\omega t} g(t)
    \quad\text{for almost every }t\in\R.
\end{equation}
In this work, we will use semi-regular Gabor systems of the form
\begin{equation}
    \cG(g,\Lambda \times \Z) = \{ M_\omega T_x g : (x,\omega) \in \Lambda \times \Z\},
\end{equation}
where the window $g \in L^2(\R)$ and $\Lambda \subset \R$ discrete. If $\Lambda$ is a lattice, i.e., $\Lambda = \delta \Z$, $\delta > 0$, then we obtain a rectangular Gabor system with time spacing $\delta$ and frequency spacing 1. For a Gabor system over a general rectangular lattice $a \Z \times b \Z$, $a,b>0$, we also use the notation
\begin{equation}
  \cG(g,a,b) := \cG(g,a\Z \times b\Z) = \{M_\omega T_x g : (x,\omega) \in a \Z \times b \Z \}.
\end{equation}
The Gabor system is a frame for $L^2(\R)$ if there are constants $0<A\leq B<\infty$ such that
\begin{equation}
  A\|f\|_{L^2}^2
  \leq
  \sum_{m,n\in\Z}|\langle f,M_{mb}T_{na}g\rangle|^2
  \leq
  B\|f\|_{L^2}^2
  \quad\text{for all }f\in L^2(\R).
\end{equation}
Thus, a Gabor frame provides a stable, generally redundant expansion of every function in $L^2(\mathbb{R})$ in terms of time-frequency shifts $M_\omega T_x g$ of a single window $g$. The frame-set problem asks which parameters $(a,b)$ satisfy this stable reconstruction property for a fixed $g$. Put differently: for which $(a,b)$ is $\cG(g,a,b)$ a frame? As in \cite{Gro14}, we define the frame set of $g$ as
\begin{equation}
  \cF (g) = \{(a,b) \in \R_{>0}^2 : \cG(g,a,b) \text{ is a frame}\}.
\end{equation}

For $n \in \N_0 = \N \cup \{0\}$, we define the $n$-th Hermite function by
\begin{equation}\label{eq:hn}
    h_n(t)
    =
    c_n(-1)^n \e^{\pi t^2}
    \frac{d^n}{dt^n} \e^{-2\pi t^2},    
    \qquad  \text{ for } t\in\R,
\end{equation}
where the positive constant $c_n$ is chosen so that $\|h_n\|_{L^2}=1$, see \cite{FollandPhase}. Note that $h_0$ is the Gaussian function $h_0(t) = 2^{1/4} \e^{-\pi t^2}$ for $t\in\R$. The Hermite function of order $n$ is thus a Gaussian multiplied by a suitably normalized Hermite polynomial.

In \cite{Gro14}, a conjecture suggesting a relatively simple structure of the frame sets of Hermite functions of order $n\geq1$ was proposed. Subsequent work disproved this conjecture for all higher-order Hermite functions, leaving only the case $n=1$ open \cite{Lem17,HorLemVid25}. The remaining conjecture, indeed formulated before \cite{Gro14}, was first stated by Lyubarskii and Nes \cite{LN}.
\begin{conjecture}[\cite{LN}, \S~5]\label{conj:h1}
    The frame set of the first Hermite function is given by
    \begin{equation}
        \cF(h_1)
        =\left\{(a,b) \in \R_{>0}^2:
        ab<1
        \text{ and }
        ab\neq\frac{q-1}{q}
        \text{ for every }q\in\Z_{\geq2}
        \right\}.
    \end{equation}
\end{conjecture}
Our main result proves that Conjecture~\ref{conj:h1} is correct.

\begin{theorem}\label{thm:main}
  Let $(a,b) \in \R_{>0}^2$. Then, 
  \begin{equation}\label{eq:mainequivalence}
    \cG(h_1,a,b)\text{ is a frame}
    \quad\Longleftrightarrow\quad
    ab<1
    \quad\text{and}\quad
    ab\neq \frac{q-1}{q}
    \quad\text{for every }q\in\Z_{\geq2}.
  \end{equation}
\end{theorem}
The proof has one central amplification step. A non-frame Gabor system produces a nonzero Gaussian shift-invariant entire function $F$ whose derivative vanishes on a translate of the lattice $\delta\Z$. For every integer $N\geq2$ for which the relevant characters $\e^{2 \pi \ii j \delta}$ with $j=0, \ldots N-1$ are pairwise distinct, we form the Wronskian of the translates $F(z-j\delta)$. These translates are linearly independent, so the Wronskian is nonzero, while their common critical points become zeros of multiplicity at least $N-1$. Automorphy shows that, after translation and dilation, the Wronskian is again a Gaussian shift-invariant entire function with bounded coefficients. A sharp zero-density theorem then excludes every irrational $\delta<1$ and forces a rational product $\delta=p/q$ in lowest terms to satisfy $p=q-1$.
\begin{remark}
    We note that \cite{LieShafTay} establishes the asserted frame-set characterization of Theorem \ref{thm:main} under the assumption that $ab\in\mathbb Q$ and, more generally, treats rational lattices beyond the rectangular setting. Conjecture~\ref{conj:h1}, however, was formulated for rectangular lattices. Our contribution is a complete proof for arbitrary rectangular lattice parameters, based on a novel, fundamentally different proof mechanism. In particular, our argument introduces Wronskian amplification as a new method for studying Gabor frame sets.
\end{remark}

The implication ``$\Longrightarrow$'' in \eqref{eq:mainequivalence} follows from three established facts. First, the density theorem for lattice Gabor frames yields $ab\leq 1$ \cite{GroechenigBook,HeilDensity}. Second, at the critical density $ab=1$, the Balian--Low theorem excludes a frame generated by the first Hermite function $h_1$ \cite{Balian,Low}; see also \cite{BHW} and \cite[Chap.~8]{GroechenigBook}. Third, Lyubarskii and Nes \cite{LN} proved that an odd window $g\in M^1(\R)$, the so-called Feichtinger algebra introduced in \cite{Fei81}, cannot generate a rectangular Gabor frame when
\begin{equation}
    ab=\frac{q-1}{q} \qquad \text{for some } q\in\Z_{\geq 2}.
\end{equation}
Since $h_1\in\mathcal S(\R)\subset M^1(\R)$, this obstruction applies to the first Hermite function. It therefore remains to prove that $\mathcal G(h_1,a,b)$ is a frame whenever
\begin{equation}
    0<ab<1 \qquad\text{and}\qquad ab\neq\frac{q-1}{q} \quad\text{for every }q\in\Z_{\geq 2}.
\end{equation}

Our contribution is the proof of the converse implication ``$\Longleftarrow$''. Set $\delta=ab$. In \S~\ref{sec:preliminaries}, we first use the characterization of semi-regular Gabor frames in terms of sampling in shift-invariant spaces \cite{GRSsampling} to show that, if the Gabor system is not a frame, then there exists a nonzero entire function of the form
\begin{equation}
    F(z)=\sum_{k\in\Z}c_k \e^{-\pi(z-k)^2/\beta}, \qquad (c_k)_{k\in\Z}\in\ell^\infty(\Z), \ \beta > 0,
\end{equation}
such that
\begin{equation}
    F'(x_*+n\delta)=0 \qquad\text{for every }n\in\Z
\end{equation}
and some $x_*\in\R$. Thus, the frame problem is reduced to determining whether a nonzero Gaussian shift-invariant entire function can have critical points on the shifted lattice $x_*+\delta \Z$.

To rule out such a function, we fix an integer $N\geq 2$ for which the characters
\begin{equation}
    1,\e^{2\pi\ii\delta},\ldots,\e^{2\pi\ii(N-1)\delta}
\end{equation} 
are pairwise distinct, and consider the Wronskian of the $N$ translates
\begin{equation}
    F(z),F(z-\delta),\ldots,F(z-(N-1)\delta),
\end{equation}
in \S~\ref{sec:Wronskian}. The distinctness of the characters implies that these translates are linearly independent, and that their Wronskian is not identically zero. On the other hand, the common critical-point condition implies that this Wronskian has a zero of multiplicity at least $N-1$ at every point of $x_*+\delta\Z$. This Wronskian construction is the main new proof ingredient: it converts common critical points of translates into zeros of increasing multiplicity.

In \S~\ref{sec:Wronskian_SI_function}, we establish that the Wronskian inherits an automorphy relation from the Gaussian shift-invariant function $F$. After a suitable translation and dilation, it is therefore again a nonzero Gaussian shift-invariant function with bounded coefficients. Its zero set contains a translate of the lattice $N\delta\Z$, with multiplicity at least $N-1$ at every point. The weighted lower density of these zeros is consequently
\begin{equation}
    \frac{N-1}{N\delta}.
\end{equation}

Lastly, in \S~\ref{sec:zero_density_proof}, we use a sharp zero-density theorem for Gaussian shift-invariant functions to bound this quantity from above by $1$, and hence
\begin{equation}
    \delta\geq 1-\frac1N.
\end{equation}
The arithmetic conclusion is now immediate. If $\delta$ is irrational, the relevant characters are pairwise distinct for every $N\geq2$, and letting $N$ tend to infinity contradicts $\delta<1$. If $\delta=p/q$ in lowest terms, we may take $N=q$, which gives
\begin{equation}
    \frac pq\geq\frac{q-1}{q}.
\end{equation} Since $p<q$, this forces $p=q-1$. Thus, failure of the frame property below the critical density can occur only when $\delta=(q-1)/q$, which gives the claimed result.

The Wronskian construction may be viewed as a multiplicity-amplification principle: it converts common critical points of translates into zeros of increasing multiplicity, while automorphy keeps the resulting function within a Gaussian shift-invariant class. This allows a sharp zero-density theorem to be applied to information that initially concerns critical points rather than zeros, and adds a new method to the study of Gabor frame sets.

\subsection{Related work}
\label{subsec:related-work-priority}

Determining the frame set of a prescribed window is an active topic in Gabor analysis, and these sets can have a surprisingly intricate geometry even for elementary windows, as evidenced by the characteristic function of an interval: its parameter region gives rise to the celebrated ``Janssen tie''~\cite{JanssenTie}, and its complete classification became known as the ``$abc$-problem for Gabor systems'', solved by Dai and Sun~\cite{DaiSunABC}. Complete descriptions have, for example, also been obtained for the hyperbolic secant~\cite{JanssenStrohmer2002}, totally positive functions of finite or Gaussian type~\cite{GRSsampling,GroechenigStoeckler2013}, the Haar window~\cite{DaiZhu2024}, and certain rational functions \cite{BelKulLyu23}. Most recently, de Dios Pont, Gr\"ochenig, Liehr, Shafkulovska, and Taylor gave a complete characterization of the frame set of all totally positive functions~\cite{DiosPontEtAl2026}.

The frame-set problem for Hermite windows lies at the intersection of Gabor analysis, sampling in Bargmann--Fock spaces, and the theory of theta functions. For a rectangular lattice Gabor system with parameters $a,b>0$, the density theorem gives the necessary condition $ab\leq 1$ \cite{HeilDensity}, while at critical density $ab=1$ the Balian--Low theorem excludes a frame generated by the first Hermite function \cite{Balian,Low}; cf.~\cite{BHW}. For the Gaussian, the full subcritical region is obtained through the sampling theory of the Bargmann--Fock space \cite{Lyubarskii1992,Seip1992,SeipWallsten1992}. Building on this connection, Gr\"ochenig and Lyubarskii developed Fock-space methods for Gabor systems generated by Hermite functions and, in particular, proved the sufficient condition $ab<1/(n+1)$ for the $n$-th Hermite function \cite{GrochenigLyubarskii2007,GrochenigLyubarskii2009}. The range of parameters $a,b>0$ for which $\cG(h_n,a,b)$ is a frame has recently been extended \cite{FauShaZlo25, FauShaZlo26}. Yet, for $n\geq2$ we lack a complete characterization of $\cF(h_n)$.

Lyubarskii and Nes \cite{LN} identified the exceptional densities relevant to the present paper.  Working at rational oversampling $1 < ab \in \Q$, they used the vector-valued Zak transform and a rational analog of the Ron--Shen Gramian \cite{RonShen1997} (see also \cite{ZibZee97}) to prove that, for every odd window in the Feichtinger algebra, the Gabor system fails to be a frame whenever
\begin{equation}
        ab=\frac{q-1}{q}  \text{ and } q\geq 2.
\end{equation}
They then conjectured that these are the only obstructions below critical density for the first Hermite function. This became the remaining first-order case in the broader conjectural picture for Hermite frame sets discussed in \cite{Gro14}. Lemvig subsequently disproved the corresponding conjecture for several infinite families of higher-order Hermite functions \cite{Lem17}; Horst, Lemvig, and Videb{\ae}k settled all of the remaining higher-order cases negatively, leaving precisely the first Hermite function as the unresolved case \cite{HorLemVid25}. Liehr, Shafkulovska, and Taylor \cite{LieShafTay} confirmed the conjecture of Lyubarskii and Nes \cite{LN} for rational lattices, i.e., when $ab \in \Q$, using the rational Zibulski--Zeevi analog \cite{ZibZee97} of the Ron--Shen Gramian \cite{RonShen1997}.

For more details on Gabor systems and frames, we refer the reader to \cite{Christensen_2016}, \cite{GroechenigBook}.

\subsection{Disclosure on the usage of LLMs}

The authors supplied ChatGPT 5.6 Pro with a detailed roadmap for proving the conjecture. This roadmap pursued an approach based on Zak transforms, theta functions, and rational-density reductions. The proof in the present paper is based on ChatGPT's response to the roadmap.  Notably, it does \emph{not} complete that roadmap; instead, it follows a different route. The semi-regular Gabor-frame characterization and the Gaussian zero-density theorem of Gröchenig, Romero, and Stöckler with translated Wronskians were not included in the initial roadmap. The authors subsequently examined, corrected, reorganized, and rewrote the argument and take full responsibility for the correctness and presentation of the final manuscript. 

\section{Preliminaries and reduction to a critical-point problem}\label{sec:preliminaries}
We use the unnormalized first Hermite function $\phi_1\colon\R\to\R$ defined by
\begin{equation}\label{eq:phi1}
  \phi_1(t)=t\e^{-\pi t^2}
  \quad t\in\R,
\end{equation}
as multiplication of the window by a nonzero constant does not affect the frame property.

The Wiener amalgam space of continuous functions, as defined in \cite[Chapter~6]{GroechenigBook}, is
\begin{equation}
W_0=\left\{g\in C(\R):
    \sum_{k\in\Z}\max_{t\in[0,1]}|g(t-k)|<\infty\right\}.
\end{equation}
For $1\leq p\leq\infty$, $g\in W_0$, the shift-invariant space generated by the integer translates of $g$ is
\begin{equation}\label{eq:shift-invariant-space}
    V^p(g)
    =\left\{
        \sum_{k\in\Z}c_k g(\cdot-k) : c=(c_k)_{k\in\Z}\in\ell^p(\Z)
    \right\}.
\end{equation}
For $p<\infty$, the series in \eqref{eq:shift-invariant-space} are understood to converge in $L^p(\R)$. For $p=\infty$, convergence is understood locally uniformly and defines a bounded continuous function.
\begin{definition}\label{def:stableIntegerShifts}
    We say $g\in W_0$ has \emph{stable integer shifts} if there are $m_g,M_g>0$ such that, for every sequence $u=(u_k)_{k\in\Z}\in\ell^2(\Z)$, the series $\sum_{k\in\Z}u_kg(\cdot-k)$ converges in $L^2(\R)$ and
    \begin{equation}\label{eq:stable-integer-shifts}
        m_g\|u\|_{\ell^2}^2
        \leq
        \left\|\sum_{k\in\Z}u_kg(\cdot-k)\right\|_{L^2}^2
        \leq
        M_g\|u\|_{\ell^2}^2.
    \end{equation}
\end{definition}
For $g\in W_0$, the upper estimate in \eqref{eq:stable-integer-shifts} follows from H\"older's inequality \cite{AldroubiGroechenig}. Moreover, \cite[Theorem~2.1]{GRSsampling} shows that the lower estimate in \eqref{eq:stable-integer-shifts} for $p=2$ is equivalent to the corresponding lower estimate in $L^p(\R)$ for some and hence for every $p\in[1,\infty]$. Therefore, Definition \ref{def:stableIntegerShifts} is equivalent to the definition of stable integer shifts used in \cite{GRSsampling}. We will also need the following notion from \cite{GRSsampling} for a set $\Lambda \subset \R$.
\begin{definition}
    A set $\Lambda\subset\R$ is \emph{relatively separated} if it contains only a uniformly bounded number of points in every interval of length one, i.e.,
    \begin{equation}
        \sup_{t\in\R}\#\bigl(\Lambda\cap[t,t+1]\bigr)<\infty.
    \end{equation}
    For $\delta>0$, the set $\Lambda$ is called \emph{$\delta$-separated} if
    \begin{equation}
        |\lambda-\mu|\geq\delta
        \quad\text{whenever }\lambda,\mu\in\Lambda\text{ and }\lambda\neq\mu.
    \end{equation}
\end{definition}
We recall Beurling's notion of a weak limit of a sequence of sets, see \cite{GRSsampling}.
\begin{definition}
    Let $\Lambda\subset\R$ and let $(\Lambda_n)_{n \in \N}$ be a sequence of subsets of $\R$. We say that $\Lambda_n$ \emph{converges weakly} to $\Lambda$ and write $\Lambda_n\xrightarrow{w}\Lambda$ if, for every $R>0$ and $\varepsilon>0$, there exists an integer $n_0\geq1$ such that, for all integers $n\geq n_0$,
    \begin{equation}
        \Lambda_n\cap[-R,R]
        \subset \Lambda+(-\varepsilon,\varepsilon)
        \quad\text{and}\quad
        \Lambda\cap[-R,R]
        \subset \Lambda_n+(-\varepsilon,\varepsilon).
    \end{equation}
    For a set $\Lambda\subset\R$, let $\Hullw(\Lambda)$ be the collection of all sets $\Gamma\subset\R$ for which there is a sequence $(t_n)_{n \in \N}$ in $\R$ such that $t_n+\Lambda\xrightarrow{w}\Gamma$.
\end{definition}

We now recall \cite[Theorem~3.3]{GRSsampling}, with the matrix injectivity condition written out explicitly in \eqref{eq:GRS-injectivity}. This result allows us to transfer the considered frame-set problem to a sampling problem in a shift-invariant space. As noted in \cite[Sec.~2.3]{GRSsampling}, the connection between Gabor duality and sampling in shift-invariant spaces is implicit in the duality theories of Janssen~\cite{Janssen1995}, and Ron and Shen \cite{RonShen1997}.

\begin{theorem}[{\cite[Theorem~3.3]{GRSsampling}}]\label{thm:GRS-frame}
    Let $g\in W_0$ have stable integer shifts, and let $\Lambda\subset\R$ be relatively separated. Then $\cG(g,(-\Lambda) \times \Z)$ is a frame for $L^2(\R)$ if and only if the following holds: for every $\Gamma\in\Hullw(\Lambda)$ and every sequence $u=(u_k)_{k\in\Z}\in\ell^\infty(\Z)$,
    \begin{equation}\label{eq:GRS-injectivity}
        \sum_{k\in\Z}u_k g(\gamma-k)=0
        \quad\text{for all }\gamma\in\Gamma
        \quad\Longrightarrow\quad u=0.
    \end{equation}
\end{theorem}

Proposition~\ref{prop:nonframe-reduction} specializes the preceding theorem to the derivative of a Gaussian and the lattice $\delta\Z$. In this setting, every weak limit of a lattice translate is again a translate of $\delta\Z$, so failure of the frame property produces a nonzero bounded coefficient sequence. More precisely, for $\varphi_\beta(t)=\e^{-\pi t^2/\beta}$, $t \in \R$, the associated Gaussian shift-invariant function belongs to $V^\infty(\varphi_\beta)$ and is given by
\begin{equation}
    F(t)=\sum_{k\in\Z}c_k\varphi_\beta(t-k),
    \quad  \text{ for } t\in\R.
\end{equation}
Its entire extension has a vanishing derivative on a shifted real lattice. Thus, the proposition reduces the original frame-set problem to a uniqueness problem for the critical points of an entire function.

\begin{proposition}\label{prop:nonframe-reduction}
    Let $a,b>0$, and set $\delta=ab$ and $\beta=b^2$. Further let $0<\delta<1$ and assume that $\cG(\phi_1,a,b)$ is not a frame. Then there exists $x_*\in\R$ and a nonzero sequence $c=(c_k)_{k\in\Z}\in\ell^\infty(\Z)$ such that the series
    \begin{equation}\label{eq:F-def}
        F(z)=\sum_{k\in\Z}c_k\e^{-\pi(z-k)^2/\beta},
        \quad  \text{ for } z\in\C,
    \end{equation} 
    converges locally uniformly on $\C$, defines a nonzero entire function, and satisfies
    \begin{equation}\label{eq:critical-grid}
        F'(x_*+n\delta)=0
        \quad\text{for all }n\in\Z.
    \end{equation} 
    \end{proposition}
    \begin{proof}
        For $t \in \R$ and $f \in L^2(\R)$, we define the unitary dilation
        \begin{equation}
            (D_bf)(t)\coloneqq b^{-1/2}f(t/b), \quad b > 0.
        \end{equation}
        Since $\delta=ab$, one has
        \begin{equation}\label{eq:dilation-intertwining}
            D_bM_{mb}T_{na}=M_mT_{n\delta}D_b
            \quad\text{for all }m,n\in\Z.
        \end{equation}
        Recall that $\phi_1(t) = t \e^{-\pi t^2}$ is the unnormalized first Hermite function defined in \eqref{eq:phi1}. Define
        \begin{equation}
            \varphi_\beta(t)=\e^{-\pi t^2/\beta},
            \qquad
            \psi_\beta(t)=\varphi_\beta'(t),
            \quad t \in \R.
        \end{equation}
        A direct calculation, using $\beta=b^2$, gives
        \begin{equation}\label{eq:dilated-phi1}
            D_b\phi_1=-\frac{b^{1/2}}{2\pi}\,\psi_\beta.
        \end{equation}
        By \eqref{eq:dilation-intertwining} and \eqref{eq:dilated-phi1}, $\cG(\phi_1,a,b)$ is a frame if and only if $\cG(\psi_\beta,\delta\Z\times\Z)$ is a frame (cf.~\cite[Chap.~9.4]{GroechenigBook}). The function $\psi_\beta$ belongs to $W_0$, because it is a polynomial times a Gaussian. Its integer shifts are stable. Indeed, with the Fourier transform convention
        \begin{equation}\label{eq:Fourier-convention}
            \widehat f(\xi)=\int_\R f(t)\e^{-2\pi\ii t\xi}\,dt
            \quad\text{for all }f\in L^1(\R)\text{ and }\xi\in\R,
        \end{equation}
        we have
        \begin{equation}\label{eq:psi-Fourier}
            \widehat{\psi_\beta}(\xi)
            =2\pi\ii\xi\,\widehat{\varphi_\beta}(\xi)
            =2\pi\ii\xi\sqrt\beta\,\e^{-\pi\beta\xi^2}
            \quad\text{for all }\xi\in\R.
        \end{equation}
        For $\xi\in\R$, define
        \begin{equation}\label{eq:psi-periodized-term}
            P_\beta(\xi)=\sum_{k\in\Z}|\widehat{\psi_\beta}(\xi+k)|^2.
        \end{equation}
        The function $P_\beta$ is one-periodic. Gaussian decay implies that the series defining $P_\beta$ in \eqref{eq:psi-periodized-term} converges uniformly on compact subsets of $\R$, and therefore $P_\beta$ is continuous. Moreover,
        \begin{equation}
            |\widehat{\psi_\beta}(\xi+k)|^2
            =4\pi^2\beta(\xi+k)^2\e^{-2\pi\beta(\xi+k)^2}
            \quad\text{for all }\xi\in\R\text{ and }k\in\Z.
        \end{equation}
        By \eqref{eq:psi-Fourier}, $|\widehat{\psi_\beta}(\xi+k)|^2$ only vanishes for $\xi=-k$; hence $P_\beta(\xi)>0$ for every $\xi\in\R$. Gaussian decay also gives a finite upper bound for $P_\beta$. By continuity, periodicity, and strict positivity, there are constants $0<m_\beta\leq M_\beta<\infty$ such that
        \begin{equation}
            m_\beta\leq P_\beta(\xi)\leq M_\beta
            \quad\text{for all }\xi\in[0,1].
        \end{equation}
        For every finitely supported sequence $u=(u_k)_{k\in\Z}$, Plancherel and periodization yield
        \begin{align}\label{eq:periodized-synthesis-norm}
            \left\|\sum_{k\in\Z}u_k \psi_\beta(\cdot-k)\right\|_{L^2}^2
            &= \int_\R \left|\sum_{k\in\Z}u_k\e^{-2\pi\ii k\xi}\right|^2 |\widehat{\psi_\beta}(\xi)|^2\,d\xi \\
            &= \int_0^1 \left|\sum_{k\in\Z}u_k\e^{-2\pi\ii k\xi}\right|^2 P_\beta(\xi)\,d\xi.
        \end{align}
        Parseval's identity, therefore, yields
        \begin{equation}\label{eq:psi-stability-estimates}
            m_\beta\|u\|_{\ell^2}^2
            \leq
            \left\|\sum_{k\in\Z}u_k\psi_\beta(\cdot-k)\right\|_{L^2}^2
            \leq
            M_\beta\|u\|_{\ell^2}^2.
        \end{equation}
        Since the finitely supported sequences are dense in $\ell^2(\Z)$, the upper estimate in \eqref{eq:psi-stability-estimates} extends the synthesis map continuously to all of $\ell^2(\Z)$. Passing to the limit in both estimates in \eqref{eq:psi-stability-estimates} proves the same inequalities for every $u\in\ell^2(\Z)$, and hence the integer shifts of $\psi_\beta$ are stable. Every weak limit of translates of $\delta\Z$ is itself a translate of $\delta\Z$. Indeed, let $\Gamma\subset\R$ and let $(x_j)_{j\geq1}$ be a sequence in $\R$ such that
        \begin{equation}\label{eq:weak-limit-Gamma}
            x_j+\delta\Z \xrightarrow{w} \Gamma.
        \end{equation}
        Write $x_j=m_j\delta+r_j$, where $m_j\in\Z$ and $r_j\in[0,\delta)$. After passing to a subsequence, we may assume that $r_j\to r$ for some $r\in[0,\delta]$. Since $m_j\delta+\delta\Z=\delta\Z$, we have
        \begin{equation}\label{eq:lattice-remainder-identity}
            x_j+\delta\Z=r_j+\delta\Z
            \quad\text{for all }j\geq1.
        \end{equation}
        It follows from \eqref{eq:lattice-remainder-identity} that
        \begin{equation}\label{eq:weak-limit-lattice}
            x_j+\delta\Z \xrightarrow{w} r+\delta\Z. 
        \end{equation}
        By \eqref{eq:weak-limit-Gamma} and \eqref{eq:weak-limit-lattice}, both $\Gamma$ and $r+\delta\Z$ are weak limits of the same subsequence. Moreover, both sets are $\delta$-separated: this is immediate for $r+\delta\Z$, while $\Gamma$ inherits this property from the translates $x_j+\delta\Z$ through the two-sided approximation in the definition of weak convergence. Hence, both sets are closed. If they were distinct, a point of one would have positive distance from the other, contradicting the fact that the same subsequence converges weakly to both. Therefore, $\Gamma=r+\delta\Z$. If $r=\delta$, then $r+\delta\Z=\delta\Z$. Hence, in every case $\Gamma$ is a translate of~$\delta\Z$.

        Since $-\delta\Z=\delta\Z$, Theorem~\ref{thm:GRS-frame}, applied to the non-frame system $\cG(\psi_\beta,\delta\Z\times\Z)$, gives a set $\Gamma\in\Hullw(\delta\Z)$ and a nonzero sequence
        $c=(c_k)_{k\in\Z}\in\ell^\infty(\Z)$ such that
        \begin{equation}\label{eq:annihilating-sequence}
            \sum_{k\in\Z}c_k \psi_\beta(\gamma-k)=0
            \quad\text{for all }\gamma\in\Gamma.
        \end{equation} 
        By the preceding paragraph, $\Gamma=x_*+\delta\Z$ for some $x_*\in\R$. Therefore,
        \begin{equation}\label{eq:annihilation-grid}
            0=\sum_{k\in\Z}c_k \psi_\beta(x_*+n\delta-k)
            \quad\text{for all }n\in\Z.
        \end{equation} 
        Finally, as $c\in\ell^\infty(\Z)$, the series in \eqref{eq:F-def} converges normally, and hence locally uniformly, on $\C$ by the Weierstrass criterion; see \cite[pp.~104--107]{Remmert}. As each summand is entire, $F$ is entire.

        To see that $F$ is nonzero, define
        \begin{equation}\label{eq:nonframe-Laurent-function}
            G(w)=\sum_{k\in\Z}c_k\e^{-\pi k^2/\beta}w^k,
            \qquad \text{ for } w\in\C^* = \C \setminus \{0\}.
        \end{equation}
        The series in \eqref{eq:nonframe-Laurent-function} converges normally on compact annuli and therefore defines a holomorphic function on $\C^*$. Completing the square gives
        \begin{equation}\label{eq:nonframe-Laurent-identity}
            \e^{\pi z^2/\beta}F(z)=G\bigl(\e^{2\pi z/\beta}\bigr),
            \qquad \text{ for }  z\in\C.
        \end{equation}
        If $F\equiv0$, the surjectivity of the exponential map $z\mapsto\e^{2\pi z/\beta}$ onto $\C^*$ would imply $G\equiv0$. Uniqueness of Laurent coefficients would then give $c_k=0$ for all $k\in\Z$, contradicting the choice of $c$. Consequently, $F\not\equiv0$.

        The series defining $F$ in \eqref{eq:F-def} may now be differentiated termwise, and the derivative converges locally uniformly on $\C$. Thus $F$ is entire and
        \begin{equation}\label{eq:F-derivative}
            F'(z)=\sum_{k\in\Z}c_k \psi_\beta(z-k)
            \quad\text{for all }z\in\C.
        \end{equation}
        Combining \eqref{eq:annihilation-grid} and \eqref{eq:F-derivative} then gives the claim
        \begin{equation} 
            F'(x_*+n\delta)
            =\sum_{k\in\Z}c_k \psi_\beta(x_*+n\delta-k)
            = 0 
            \quad\text{for all }n\in\Z,
        \end{equation}
    \end{proof}

\section{Wronskian amplification}\label{sec:Wronskian}
Throughout this section, let $\beta>0$, $0<\delta<1$, and let $c=(c_k)_{k\in\Z}\in\ell^\infty(\Z)$ be nonzero. We denote by $F$ the entire function defined in \eqref{eq:F-def}. We aim to amplify the critical-point condition \eqref{eq:critical-grid} by forming the Wronskian of finitely many translates of $F$.

The argument has two separate components. First, the translates are eigenfunctions of a common weighted translation operator. Distinctness of the corresponding eigenvalues implies their linear independence and hence the nonvanishing of their Wronskian. Second, the common critical points of the translates force this Wronskian to vanish to high order. More precisely, for every integer $N\geq2$ for which the corresponding characters are pairwise distinct, we obtain a nonzero Wronskian that has a zero of multiplicity at least $N-1$ at every point of the shifted lattice $x_*+\delta\Z$. We refer to this construction as \emph{Wronskian amplification}.

\begin{definition}
    Let $N \in \N$, let $\Omega\subset\C$ be a connected open set, and let $f_0,\ldots,f_{N-1}$ be holomorphic on $\Omega$. Their Wronskian matrix is
    \begin{equation}\label{eq:general-wronskian-matrix}
        M = \begin{pmatrix}
            f_0 & \ldots & f_{N-1}\\
            f_0' & \ldots & f_{N-1}'\\
            \vdots & \ddots & \vdots\\
            f_0^{(N-1)} & \ldots & f_{N-1}^{(N-1)}
        \end{pmatrix}
    \end{equation}
    The \emph{Wronskian} is the determinant of the matrix $M$ in \eqref{eq:general-wronskian-matrix}, i.e., the holomorphic function
    \begin{equation}\label{eq:general-wronskian}
        \Wr(f_0,\ldots,f_{N-1})(z)
        =\det\left(f_j^{(r)}(z)\right)_{0\leq r,j<N}
        \quad\text{for all }z\in\Omega.
    \end{equation}
\end{definition}

We recall the analytic Wronskian criterion; see \cite[\S~4]{Bocher} and, for a modern proof, \cite{BostanDumas}.
\begin{theorem}\label{thm:wronskian-criterion}
    Let $N \in \N$, $\Omega\subset\C$ a connected open set. Let $f_0,\ldots,f_{N-1}$ be holomorphic on $\Omega$. Then the functions $f_0,\ldots,f_{N-1}$ are linearly dependent over $\C$ if and only if
    \begin{equation}
        \Wr(f_0,\ldots,f_{N-1})\equiv0
        \quad\text{on }\Omega.
    \end{equation}
    Equivalently, they are linearly independent over $\C$ if and only if their Wronskian does not vanish identically.
\end{theorem}
We note that analyticity is essential. For functions that are merely smooth, an identically vanishing Wronskian need not imply linear dependence; see again \cite[\S~4]{Bocher} and \cite{BostanDumas}.

We now specialize the Wronskian construction to the Gaussian shift-invariant function $F$ from \eqref{eq:F-def}. For an integer $N\geq2$ and every $j\in\{0,\ldots,N-1\}$, define the entire function
\begin{equation}\label{eq:F-translates}
    F_j(z):=F(z-j\delta)
    \quad\text{for all }z\in\C,
\end{equation}
and denote the Wronskian of the translates $F_j$ of the Gaussian shift-invariant function $F$ by
\begin{equation}\label{eq:wronskian}
    \cW_N:=\Wr(F_0,\ldots,F_{N-1}).
\end{equation}
Equivalently,
\begin{equation}
    \cW_N(z)=\det\left(F_j^{(r)}(z)\right)_{0\leq r,j<N}
    \quad\text{for all }z\in\C.
\end{equation}

\begin{lemma}\label{lem:analytic-F}
    Let $c = (c_k)_{k\in \Z}\in\ell^\infty(\Z)$ be nonzero, and let $F$ be defined as in \eqref{eq:F-def}. Then the following conditions hold.
    \begin{enumerate}[(i)]
        \item\label{item:F_convergence}
        The series defining $F$, as well as any series obtained by differentiating it termwise any finite number of times, converge locally uniformly on $\C$. In particular, $F$ is entire and nonzero.
        \item\label{item:F_automorphy}
        The function $F$ satisfies the automorphy relation
        \begin{equation}\label{eq:quasiperiod-F}
          F(z+\ii\beta)=\e^{\pi\beta-2\pi\ii z}F(z)
          \eqqcolon \rho_\beta(z) F(z)
          \quad\text{for all }z\in\C.
        \end{equation}
        \item\label{item:F_bounded}
        For every integer $r\geq0$,
        \begin{equation}\label{eq:strip-bound-F}
          \sup_{\substack{x\in\R\\0\leq y\leq\beta}}
          |F^{(r)}(x+\ii y)|<\infty.
        \end{equation}
    \end{enumerate}
\end{lemma}

\begin{proof}
    We show \eqref{item:F_convergence}. Every derivative of a Gaussian is a polynomial times the same Gaussian. Fix an integer $r\geq0$. On each compact subset of $\C$, the $k$-th summand and its $r$-th derivative are therefore bounded, for suitable constants $C,\gamma>0$, by a summable sequence of the form $C(1+|k|)^r\e^{-\gamma k^2}$. This proves normal convergence and term-by-term differentiation. Using the arguments surrounding \eqref{eq:nonframe-Laurent-function} for general $c$, we conclude that $F\equiv0$ would imply $c=0$. Hence, $F\not\equiv0$.
    
    We now prove \eqref{item:F_automorphy}. A direct computation shows that
    \begin{equation}\label{eq:Gaussian-summand-quasiperiodicity}
        \e^{-\pi(z+\ii\beta-k)^2/\beta}
        =\e^{\pi\beta-2\pi\ii z}\e^{2\pi\ii k}
        \e^{-\pi(z-k)^2/\beta}
        \quad\text{for all }z\in\C\text{ and }k\in\Z.
    \end{equation}
    Summing \eqref{eq:Gaussian-summand-quasiperiodicity} over $k\in\Z$ gives \eqref{eq:quasiperiod-F}.
    
    Lastly, we prove \eqref{item:F_bounded}. For $x\in\R$ and $0\leq y\leq\beta$, each differentiated summand is bounded, for a suitable constant $C_{r,\beta}>0$, by
    \begin{equation}\label{eq:differentiated-Gaussian-bound}
        \|c\|_{\ell^\infty}C_{r,\beta}
        (1+|x-k|)^r\e^{-\pi(x-k)^2/\beta}.
    \end{equation}
    Summing over $k\in\Z$ gives a continuous $1$-periodic majorant as a function of $x$. It is therefore bounded on $\R$, which proves \eqref{eq:strip-bound-F}.
\end{proof}

The automorphy relation in \eqref{eq:quasiperiod-F} can be written using the weighted vertical shift $V_\beta$ on entire functions, defined by
\begin{equation}\label{eq:weighted-vertical-translation}
    (V_\beta H)(z)=\e^{-\pi\beta+2\pi\ii z}H(z+\ii\beta)
    = \rho_\beta(z)^{-1} H(z+\ii\beta)
    \quad\text{for all }z\in\C.
\end{equation}
The prefactor $\rho_\beta(z)^{-1}$ in \eqref{eq:weighted-vertical-translation} cancels the common automorphy factor $\rho_\beta(z)$ in \eqref{eq:quasiperiod-F}. Consequently, with $\chi_j=\e^{2\pi\ii j\delta}$,
\begin{equation}\label{eq:V-eigenfunctions}
  V_\beta F_j=\chi_jF_j
  \quad\text{for all }j\in\{0,\ldots,N-1\}.
\end{equation}
Thus, the functions $F_j$ are eigenfunctions of a single operator, and the distinctness of their characters will give their linear independence. The operator $V_\beta$ will be used only for this linear-independence argument.

\begin{remark}
    The vertical-shift operator $V_\beta$ is precisely a Bargmann--Fock weighted translation $\tau_\zeta$ in the normalization
    \begin{equation}\label{eq:Fock-weighted-translation}
        (\tau_\zeta H)(z)
        =\e^{\alpha z\overline\zeta-\alpha|\zeta|^2/2}H(z-\zeta),
        \quad \text{for } z, \zeta \in \C, \, \alpha>0. 
    \end{equation}
    In \eqref{eq:Fock-weighted-translation}, taking $\alpha=2\pi/\beta$ and $\zeta=-\ii\beta$ gives $\tau_\zeta=V_\beta$; see \cite[Chapter~1, Section~6]{FollandPhase} for the Bargmann--Fock representation and \cite[Chapter~2]{ZhuFock} for weighted translations. For the related Heisenberg-group and theta-function automorphy viewpoint, see \cite[Chapter~I]{MumfordTheta}.
\end{remark}

For a function $H$ that is holomorphic in a neighborhood of $\lambda$ and not identically zero, we write $\ord_\lambda H$ for the order of vanishing of $H$ at $\lambda$. Thus $\ord_\lambda H$ is the unique integer $m\geq0$ for which there exist a neighborhood $U$ of $\lambda$ and a function $h$ holomorphic on $U$ such that
\begin{equation}
    H(z)=(z-\lambda)^m h(z),
    \qquad \text{ for all } z\in U,
\end{equation}
and $h(\lambda)\neq0$. Equivalently, $\ord_\lambda H$ is the smallest $m\geq0$ such that $H^{(m)}(\lambda)\neq0$.

Lemma~\ref{lem:wronskian-amplification} below is a special case of the classical ramification-weight formula for Wronskians of finite-dimensional linear systems, see \cite[Theorem~15(ii)]{Laksov1984} and \cite[\S3.12]{GattoScherbak2012}. We give a direct proof adapted to the hypothesis that functions have common critical points used here.  The determinants $D(r_0,\ldots,r_{N-1};z)$ used in the proof below are generalized Wronskians, up to a permutation of their rows. The expansion of a derivative of an ordinary Wronskian into such determinants is standard, see, e.g.,\cite[Proposition~1]{Towse2000}.

\begin{lemma}
    \label{lem:wronskian-amplification}
    Let $N \in \N$, $N \geq 2$, let $U\subset\C$ be a connected open neighborhood of $\lambda\in\C$,
    and let $f_0,\ldots,f_{N-1}$ be holomorphic on $U$. Denote the Wronskian by
    \begin{equation}\label{eq:Wr-amplification-W}
        \cW(z)
        := \Wr(f_0,\ldots,f_{N-1})(z)
        = \det\bigl(f_j^{(r)}(z)\bigr)_{0\leq r,j<N},
        \qquad \text{ for } z\in U.
    \end{equation}
    Suppose that
    \begin{equation}\label{eq:Wr-amplification-hypothesis}
        f_j'(\lambda)=0,
        \qquad \text{ for all } j=0,\ldots,N-1.
    \end{equation}
    Then
    \begin{equation}\label{eq:Wr-derivative-vanishing}
        \cW^{(m)}(\lambda)=0,
        \qquad \text{ for all } m=0,\ldots,N-2.
    \end{equation}
    Consequently, if $\cW\not\equiv 0$, then
    \begin{equation}\label{eq:Wr-order}
        \ord_{\lambda}\cW\geq N-1.
    \end{equation}
\end{lemma}

\begin{proof}
    For nonnegative integers $r_0,\ldots,r_{N-1}$ and $z\in U$, set
    \begin{equation}\label{eq:D-determinant}
        D(r_0,\ldots,r_{N-1};z)
        :=
        \det\bigl(f_j^{(r_k)}(z)\bigr)_{0\leq k,j<N}.
    \end{equation}
    Thus
    \begin{equation}\label{eq:W-as-D}
        \cW(z)=D(0,1,\ldots,N-1;z).
    \end{equation}
    The point of the underlying argument now is that, after fewer than $N-1$ differentiations, every determinant arising from the Leibniz rule either has two equal rows or contains the row of first derivatives, which vanishes at $\lambda$.
    
    We start with repeated differentiation of \eqref{eq:W-as-D}, which, together with the multilinearity of the determinant in its rows, gives
    \begin{align}\label{eq:expansionOfWm}
        \cW^{(m)}(z)
        =
        \sum_{\substack{\alpha\in\N_0^N\\ |\alpha|=m}}
        \binom{m}{\alpha}
        D\bigl(
            \alpha_0,
            1+\alpha_1,
            \ldots,
            N-1+\alpha_{N-1};
            z
        \bigr), \quad \text{ for all } z \in U,
    \end{align}
    where
    \begin{equation}
        |\alpha|:=\alpha_0+\cdots+\alpha_{N-1},
        \qquad
        \binom{m}{\alpha}
        :=
        \frac{m!}{\alpha_0!\cdots\alpha_{N-1}!}.
    \end{equation}
    
    Fix $m\in\{0,\ldots,N-2\}$ and consider one term in the sum \eqref{eq:expansionOfWm}. Write
    \begin{equation}\label{eq:row-indices}
        r_k:=k+\alpha_k,
        \qquad \text{ for all } k=0,\ldots,N-1.
    \end{equation}
    If two of the integers $r_0,\ldots,r_{N-1}$ coincide, then the corresponding determinant in \eqref{eq:expansionOfWm}, defined in \eqref{eq:D-determinant}, has two equal rows and therefore vanishes.
    
    Suppose instead that $r_0,\ldots,r_{N-1}$ are pairwise distinct. If none of them were equal to $1$, then they would be $N$ distinct nonnegative integers avoiding $1$. Their sum would therefore satisfy
    \begin{equation}\label{eq:row-index-lower-bound}
        \sum_{k=0}^{N-1}r_k
        \geq
        0+2+3+\cdots+N
        =
        \frac{N(N-1)}{2}+N-1.
    \end{equation}
    On the other hand,
    \begin{equation}\label{eq:row-index-upper-bound}
        \sum_{k=0}^{N-1}r_k
        =
        \sum_{k=0}^{N-1}k+\sum_{k=0}^{N-1}\alpha_k
        =
        \frac{N(N-1)}{2}+m
        \leq
        \frac{N(N-1)}{2}+N-2,
    \end{equation}
    which contradicts \eqref{eq:row-index-lower-bound}. Hence, one of the row indices $r_k$ defined in \eqref{eq:row-indices} must equal $1$.
    
    The row whose index equals $1$ is
    \begin{equation}\label{eq:first-derivative-row}
        \bigl(f_0'(\lambda),\ldots,f_{N-1}'(\lambda)\bigr).
    \end{equation}
    By \eqref{eq:Wr-amplification-hypothesis}, we have that \eqref{eq:first-derivative-row} is zero. Thus every term in the expansion \eqref{eq:expansionOfWm} of $\cW^{(m)}(\lambda)$ vanishes, proving \eqref{eq:Wr-derivative-vanishing}. If $W\not\equiv0$, the vanishing of these derivatives gives \eqref{eq:Wr-order}, i.e., $\ord_\lambda \cW \geq N-1$.
\end{proof}

\begin{lemma}
\label{lem:independence-multiplicity}
Let $N \in \N$, $N \geq 2$ and assume that the characters in \eqref{eq:V-eigenfunctions} satisfy
\begin{equation}\label{eq:distinct-characters}
  \chi_0,\chi_1,\ldots,\chi_{N-1}
  \quad\text{are pairwise distinct}.
\end{equation}
Then
\begin{equation}\label{eq:alternative-W-nonzero}
    \cW_N\not\equiv0.
\end{equation}
If, in addition, there exists $x_*\in\R$ such that
\eqref{eq:critical-grid} holds, then
\begin{equation}\label{eq:alternative-W-multiplicity}
    \ord_{x_*+n\delta}\cW_N
    \geq N-1,
    \qquad \text{ for all } n\in\Z.
\end{equation}
\end{lemma}

\begin{proof}
By \eqref{eq:V-eigenfunctions}, each $F_j$ is an eigenfunction of $V_\beta$ with character $\chi_j$. Each $F_j$ is nonzero, and the characters are pairwise distinct by assumption. Suppose that 
\begin{equation}
    \sum_{k=0}^{N-1}a_kF_k=0. 
\end{equation}
Fix $j\in\{0,\ldots,N-1\}$ and apply
\begin{equation}
    \prod_{\substack{0\leq\ell<N \\ \ell\neq j}} (V_\beta-\chi_\ell I)
\end{equation}
to this identity. Since $V_\beta F_k=\chi_kF_k$, all terms with $k\neq j$ vanish and the remaining term is
\begin{equation}
    a_j \prod_{\substack{0\leq\ell<N \\ \ell\neq j}} (\chi_j-\chi_\ell)F_j.
\end{equation}
The characters are pairwise distinct and $F_j$ is nonzero, so $a_j=0$. As $j$ was arbitrary, $F_0,...,F_{N-1}$ are linearly independent. The analytic Wronskian criterion, therefore, gives
\begin{equation}\label{eq:alternative-Wronskian-nonzero}
    \cW_N
    =
    \Wr(F_0,\ldots,F_{N-1})
    \not\equiv0.
\end{equation}

Now fix $n\in\Z$ and set
\begin{equation}
    \lambda:=x_*+n\delta.
\end{equation}
For every $j\in\{0,\ldots,N-1\}$,
\begin{equation}\label{eq:Fj-critical-derivative}
    F_j'(\lambda)
    =
    F'(\lambda-j\delta)
    =
    F'\bigl(x_*+(n-j)\delta\bigr)
    =
    0.
\end{equation}
Here, the last equality follows from \eqref{eq:critical-grid}. Lemma~\ref{lem:wronskian-amplification}, applied to $f_j=F_j$ using \eqref{eq:Fj-critical-derivative}, therefore yields
\begin{equation}\label{eq:alternative-local-multiplicity}
    \ord_{\lambda}\cW_N\geq N-1.
\end{equation}
Since $n\in\Z$ was arbitrary, $\cW_N$ has a zero of multiplicity at least $N-1$ at every point of the shifted lattice $x_*+\delta\Z$, proving \eqref{eq:alternative-W-multiplicity}.
\end{proof}

\section{The Wronskian as a Gaussian shift-invariant function}\label{sec:Wronskian_SI_function}
We now show that the Wronskian constructed in the preceding section retains the Gaussian shift-invariant structure needed for the zero-density argument. The automorphy relation \eqref{eq:quasiperiod-F} of $F$ induces a corresponding relation for $\cW_N$. Together with a uniform bound on a fundamental strip, this relation yields an expansion of $\cW_N$ in Gaussian translates. After a suitable translation and dilation, the Wronskian is therefore a Gaussian shift-invariant entire function with bounded coefficients.

For the remainder of this section, set
\begin{equation}\label{eq:sigma-definition}
    \sigma:=\frac{N-1}{2}\delta.
\end{equation}
\begin{lemma}\label{lem:W-automorphy}
    Let $N \in \N$, $N \geq 2$, $\delta \in (0,1)$, and $\beta >0$. Then, for every $z\in\C$, the Wronskian $\cW_N$ satisfies
    \begin{equation}\label{eq:W-automorphy}
        \cW_N(z+\ii\beta)
        = \rho_{\beta}(z-\sigma)^N \cW_N(z)
        =\e^{\pi N\beta-2\pi\ii N(z-\sigma)}\cW_N(z),
    \end{equation}
    and
    \begin{equation}\label{eq:W-strip-bound}
        \sup_{\substack{x\in\R\\0\leq y\leq\beta}}|\cW_N(x+\ii y)|<\infty.
    \end{equation}
\end{lemma}

\begin{proof}
    Recall that for $j\in\{0,\ldots,N-1\}$, the characters in \eqref{eq:V-eigenfunctions} are explicitly given by
    \begin{equation}\label{eq:characters}
        \chi_j=\e^{2\pi\ii j\delta}.
    \end{equation}
    For $j=0,\ldots,N-1$, the automorphy relation \eqref{eq:quasiperiod-F} for $F_j$ reads
    \begin{equation}
        F_j(z+\ii\beta) = \rho_\beta(z)\chi_jF_j(z),
    \end{equation}
    Differentiating $r$ times gives
    \begin{equation} 
        F_j^{(r)}(z+\ii\beta) = \rho_\beta(z)\chi_j \sum_{s=0}^r \binom{r}{s}(-2\pi\ii)^{r-s}F_j^{(s)}(z).
    \end{equation}
    Thus, $\cW_N(z+\ii\beta)$ is obtained from $\cW_N(z)$ by multiplication in the derivative index by a lower triangular matrix with diagonal entries equal to $1$, together with multiplication of the $j$-th column by $\rho_\beta(z)\chi_j$. Taking determinants yields
    \begin{equation}
        \cW_N(z+\ii\beta) = \rho_\beta(z)^N \prod_{j=0}^{N-1}\chi_j \, \cW_N(z).
    \end{equation}
    Since
    \begin{equation}\label{eq:character-product}
        \prod_{j=0}^{N-1}\chi_j
        = \e^{2\pi\ii\delta\sum_{j=0}^{N-1}j}
        =\e^{2\pi\ii\delta\binom N2}
        =\e^{2\pi\ii N\sigma}.
    \end{equation}
    we obtain
    \begin{equation}\label{eq:W-automorphy-intermediate}
        \cW_N(z+\ii\beta)
        =\rho_\beta(z)^N\prod_{j=0}^{N-1}\chi_j\,\cW_N(z)
        \quad\text{for all }z\in\C.
    \end{equation}
    Substituting \eqref{eq:characters}, \eqref{eq:character-product}, and \eqref{eq:sigma-definition} into \eqref{eq:W-automorphy-intermediate} gives \eqref{eq:W-automorphy}. The strip bound follows from the determinant formula \eqref{eq:wronskian} for $\cW_N$ and the bounds in \eqref{eq:strip-bound-F}: every entry of $\cW_N$ is uniformly bounded on $\{x+\ii y:x\in\R,\ 0\leq y\leq\beta\}$, and hence so is its determinant, proving \eqref{eq:W-strip-bound}.
\end{proof}

The next lemma characterizes the entire functions satisfying the preceding automorphy relation and a uniform strip bound.
\begin{lemma}\label{lem:automorphy-shift-invariant}
    Let $H$ be entire, let $\beta>0$, $N \in \N$, and $\sigma\in\R$. Assume that $H$ satisfies
    \begin{equation}\label{eq:H-automorphy}
        H(z+\ii\beta)
        =\e^{\pi N\beta-2\pi\ii N(z-\sigma)}H(z)
        = \rho_\beta(z-\sigma)^N H(z)
        \quad\text{for all }z\in\C,
    \end{equation}
    as well as the bound
    \begin{equation}\label{eq:H-strip-bound}
        M:=\sup_{\substack{x\in\R\\0\leq y\leq\beta}}|H(x+\ii y)|<\infty.
    \end{equation}
    Then there exists a sequence $\eta=(\eta_n)_{n\in\Z}\in\ell^\infty(\Z)$, with $\|\eta\|_{\ell^\infty} \leq M$, such that
    \begin{equation}\label{eq:H-shift-invariant}
        H(z)=\sum_{n\in\Z}\eta_n
        \exp\left[-\frac{\pi N}{\beta}
        \left(z-\sigma-\frac nN\right)^2\right]
        \quad\text{for all }z\in\C.
    \end{equation}
    The series in \eqref{eq:H-shift-invariant} converges locally uniformly on $\C$.
\end{lemma}

\begin{proof}
    Set
    \begin{equation}\label{eq:K-definition}
        K(z)=\e^{\pi N(z-\sigma)^2/\beta}H(z) \quad\text{for all }z\in\C.
    \end{equation}
    Equations \eqref{eq:K-definition} and \eqref{eq:H-automorphy} give $K(z+\ii\beta)=K(z)$ for every $z\in\C$.

    Thus, $K$ is entire and $\ii\beta$-periodic. Equivalently, $K(z)$ is a holomorphic function of $\e^{2\pi z/\beta}$ on $\C^\ast$, and hence has a Laurent expansion. Written in terms of $z$, this gives
    \begin{equation}\label{eq:K-Fourier}
        K(z)=\sum_{n\in\Z} \kappa_n\e^{2\pi n(z-\sigma)/\beta} \quad\text{for all }z\in\C,
    \end{equation}
    with locally uniform convergence on $\C$. For every $u\in\R$ and $n\in\Z$, the Fourier coefficient formula applied to the $\beta$-periodic function $y\mapsto K(u+\ii y)$ gives
    \begin{equation}\label{eq:kappa-Fourier}
        \kappa_n\e^{2\pi n(u-\sigma)/\beta} = \frac1\beta\int_0^\beta K(u+\ii y)\e^{-2\pi\ii ny/\beta}\,dy. 
    \end{equation}
    By \eqref{eq:K-definition} and \eqref{eq:H-strip-bound},
    \begin{equation}\label{eq:K-strip-estimate}
        |K(u+\ii y)| \leq M\exp\left[ \frac{\pi N}{\beta} \bigl((u-\sigma)^2-y^2\bigr) \right] 
    \end{equation}
    for every $u\in\R$ and $y\in[0,\beta]$. Moreover,
    \begin{equation}\label{eq:Gaussian-integral-bound}
        \frac1\beta\int_0^\beta \e^{-\pi Ny^2/\beta}\,dy \leq 1,
    \end{equation}
    since the integrand is bounded above by $1$ on $[0,\beta]$. Taking absolute values in \eqref{eq:kappa-Fourier} and using \eqref{eq:K-strip-estimate} and \eqref{eq:Gaussian-integral-bound}, we obtain
    \begin{equation}\label{eq:kappa-intermediate-bound}
        |\kappa_n| \leq M\exp\left[ \frac{\pi N}{\beta}(u-\sigma)^2 -\frac{2\pi n}{\beta}(u-\sigma) \right],
    \end{equation}
    for every $u\in\R$, $n\in\Z$. For fixed $n$, the exponent on the right side of \eqref{eq:kappa-intermediate-bound} is minimized at 
    \begin{equation}
        u=\sigma+\frac{n}{N}. 
    \end{equation} 
    With this choice, \eqref{eq:kappa-intermediate-bound} gives 
    \begin{equation}\label{eq:kappa-Gaussian-bound}
        |\kappa_n| \leq M\e^{-\pi n^2/(\beta N)} \quad\text{for all }n\in\Z.
    \end{equation} 
    For every $n\in\Z$, define
    \begin{equation}\label{eq:eta-definition} 
        \eta_n=\kappa_n\e^{\pi n^2/(\beta N)}.
    \end{equation} 
    By \eqref{eq:kappa-Gaussian-bound} and \eqref{eq:eta-definition}, we have $\|\eta\|_{\ell^\infty}\leq M$. Combining \eqref{eq:K-definition} and \eqref{eq:K-Fourier} gives 
    \begin{equation}\label{eq:H-from-K-Fourier}
        H(z)=\sum_{n\in\Z}\kappa_n \exp\left[ -\frac{\pi N}{\beta}(z-\sigma)^2 +\frac{2\pi n}{\beta}(z-\sigma) \right] \quad\text{for all }z\in\C.
    \end{equation} 
    For every $n\in\Z$ and $z\in\C$, the exponent in \eqref{eq:H-from-K-Fourier} satisfies
    \begin{equation}\label{eq:complete-square-H}
        -\frac{\pi N}{\beta}(z-\sigma)^2 +\frac{2\pi n}{\beta}(z-\sigma) = -\frac{\pi N}{\beta} \left(z-\sigma-\frac nN\right)^2 +\frac{\pi n^2}{\beta N}. 
    \end{equation} 
    Substituting \eqref{eq:complete-square-H} into \eqref{eq:H-from-K-Fourier} and using \eqref{eq:eta-definition} gives \eqref{eq:H-shift-invariant}. Finally, the boundedness of $\eta$ and Gaussian decay imply that the series in \eqref{eq:H-shift-invariant} converges locally uniformly on $\C$. \end{proof}

The following result shows that the Wronskian $\cW_N$ is a Gaussian shift-invariant function.
\begin{corollary}\label{cor:W-shift-invariant}
    Let $\beta>0$, $\delta\in(0,1)$, and $N\geq2$, and set
    \begin{equation}
        \sigma=\frac{N-1}{2}\delta.
    \end{equation}
    There exists a sequence $\eta=(\eta_n)_{n\in\Z}\in\ell^\infty(\Z)$ such that
    \begin{equation}
        \cW_N(z) = \sum_{n\in\Z}\eta_n \exp\left( -\frac{\pi N}{\beta} \left(z-\sigma-\frac nN\right)^2 \right)
        \quad \text{for all } z \in \C.
    \end{equation}
    Consequently, the rescaled function
    \begin{equation}\label{eq:rescaled-W}
        H_N(z):= \cW_N\left(\frac zN+\sigma\right)
    \end{equation}
    has the representation
    \begin{equation}
        H_N(z) = \sum_{n\in\Z}\eta_n \exp\left( -\frac{\pi}{\beta N}(z-n)^2 \right)
        \quad \text{for all } z \in \C.
    \end{equation}
    In particular, $H_N$ is a Gaussian shift-invariant entire function with bounded coefficients. \end{corollary}

\begin{proof}
Apply Lemma~\ref{lem:automorphy-shift-invariant} to Lemma~\ref{lem:W-automorphy}, then substitute $z/N+\sigma$ for the argument.
\end{proof}

\section{Zero density and completion of the proof}\label{sec:zero_density_proof}

For a discrete set $\Lambda\subset\R$ and a weight $\nu\colon\Lambda\to\N$, the weighted lower Beurling density is
\begin{equation} \label{eq:weighted-density}
    D^-(\Lambda,\nu)
    :=
    \liminf_{R\to\infty}\inf_{y\in\R}
    \frac{1}{2R}
    \sum_{\lambda\in\Lambda\cap[y-R,y+R]}\nu(\lambda).
\end{equation}

We use the following sharp bound on the density of zeros of Gaussian shift-invariant functions. It is a special case of \cite[Theorem~4.3]{GRSderivatives}, which also permits complex coefficients.

\begin{theorem}\label{thm:zero-density}
    Let $\beta>0$, let $u=(u_k)_{k\in\Z}\in\ell^\infty(\Z)$, and suppose that
    \begin{equation}
        f(t)
        =
        \sum_{k\in\Z}u_k\e^{-\pi(t-k)^2/\beta},
        \qquad t\in\R,
    \end{equation}
    is not identically zero. Let $\cZ(f)$ be the set of real zeros of the entire extension of $f$, and assign to each $\lambda\in\cZ(f)$ the weight
    \begin{equation}
        \nu_f(\lambda):=\ord_\lambda f.
    \end{equation}
    Then
    \begin{equation}
        D^-(\cZ(f),\nu_f)\leq1.
    \end{equation}
\end{theorem}

We now combine Wronskian amplification with Theorem~\ref{thm:zero-density}. The resulting obstruction is the analytic heart of the argument. It shows that a nonzero Gaussian shift-invariant entire function with bounded coefficients cannot have its derivative vanish on a translated copy of $\delta\mathbb Z$ below critical density, except at the arithmetic values $\delta=(q-1)/q$.

\begin{proposition}\label{prop:critical}
    Let $\beta>0$, $0<\delta<1$, $x_*\in\R$, and $c\in\ell^\infty(\Z)$. If $F$ is as in \eqref{eq:F-def}, i.e,
    \begin{equation}
        F(z) = \sum_{k \in \Z} c_k \e^{-\pi(z-k)^2/\beta}, \quad z \in \C,
    \end{equation}
    and satisfies \eqref{eq:critical-grid}, i.e,
    \begin{equation}
        F'(x_*+n\delta)=0, \quad \text{for all } n \in \Z,
    \end{equation}
    then either $c=0$, or
    \begin{equation}
        \delta=\frac{q-1}{q}, \quad q \in \Z_{\geq 2}.
    \end{equation}
\end{proposition}

\begin{proof}
    Assume that $c\neq0$, and let $N\geq2$ be an integer for which the distinctness condition \eqref{eq:distinct-characters} holds. By Lemma~\ref{lem:independence-multiplicity}, the Wronskian $\cW_N$ is not identically zero and satisfies \eqref{eq:alternative-W-multiplicity}. Consequently, the nonzero Gaussian shift-invariant function $H_N$ defined in \eqref{eq:rescaled-W} has a zero of multiplicity at least $N-1$ at every point of
    \begin{equation}\label{eq:rescaled-zero-lattice}
        \Lambda_N
        :=
        N(x_*-\sigma)+N\delta\Z.
    \end{equation}
    
    Define the constant weight $\nu_N\colon\Lambda_N\to\N$ by
    \begin{equation}
        \nu_N(\lambda):=N-1.
    \end{equation}
    Since the points of $\Lambda_N$ have spacing $N\delta$, its weighted lower Beurling density is
    \begin{equation}\label{eq:rescaled-zero-density}
        D^-(\Lambda_N,\nu_N)
        =
        \frac{N-1}{N\delta}.
    \end{equation}
    Moreover, $\Lambda_N\subset\cZ(H_N)$ and $\nu_N(\lambda)\leq\nu_{H_N}(\lambda)$ for every $\lambda\in\Lambda_N$. It follows directly from the definition of the weighted lower Beurling density that
    \begin{equation}
        D^-(\Lambda_N,\nu_N)
        \leq
        D^-(\cZ(H_N),\nu_{H_N}).
    \end{equation}
    Applying Theorem~\ref{thm:zero-density} to the representation \eqref{eq:rescaled-W}, whose Gaussian parameter is $\beta N$, yields
    \begin{equation}\label{eq:key-inequality}
        \frac{N-1}{N\delta}
        =
        D^-(\Lambda_N,\nu_N)
        \leq
        D^-(\cZ(H_N),\nu_{H_N})
        \leq1.
    \end{equation}
    Equivalently,
    \begin{equation}
        \delta\geq1-\frac1N.
    \end{equation}
    
    If $\delta$ is irrational, the characters in \eqref{eq:distinct-characters} are pairwise distinct for every $N\geq2$. Letting $N\to\infty$ in \eqref{eq:key-inequality} gives $\delta\geq1$, contradicting the assumption $\delta<1$.
    
    Suppose now that $\delta$ is rational, and write
    \begin{equation}
        \delta=\frac pq
    \end{equation}
    in lowest terms, where $p,q\in\Z_{>0}$. Taking $N=q$, the characters
    \begin{equation}
        \e^{2\pi\ii jp/q},
        \qquad j=0,\ldots,q-1,
    \end{equation}
    are the $q$ distinct roots of unity. Hence \eqref{eq:key-inequality} gives
    \begin{equation}
        \frac pq\geq\frac{q-1}{q},
    \end{equation}
    and therefore $p\geq q-1$. Since $\delta<1$ implies $p<q$, we must have $p=q-1$. Thus, as claimed,
    \begin{equation}
        \delta=\frac{q-1}{q}, \quad q \in \Z_{\geq 2}.
    \end{equation}
\end{proof}

\begin{proof}[Proof of Theorem~\ref{thm:main}]
    As explained in Section~\ref{sec:intro}, the implication ``$\Longrightarrow$'' in \eqref{eq:mainequivalence} follows from the density theorem, the Balian--Low theorem, and the non-frame result for odd windows~\cite{LN}.
    
    We now prove the implication ``$\Longleftarrow$''. Consider the Gabor system $\cG(\phi_1,a,b)$ with $\phi_1(t)$ defined in \eqref{eq:phi1}, $a,b > 0$, and suppose that
    \begin{equation}\label{eq:converse-assumption}
        0<\delta=ab<1,
        \qquad
        \delta\neq\frac{q-1}{q}
        \quad\text{for every }q\in\Z_{\geq2}.
    \end{equation}
    If $\cG(\phi_1,a,b)$ were not a frame, Proposition~\ref{prop:nonframe-reduction} would produce a nonzero sequence $c\in\ell^\infty(\Z)$ such that the function $F$ in \eqref{eq:F-def} would satisfy \eqref{eq:critical-grid}. Proposition~\ref{prop:critical} would then imply that
    \begin{equation}
        ab = \delta=\frac{q-1}{q}
    \end{equation}
    for some integer $q\geq2$, contradicting \eqref{eq:converse-assumption}. Therefore $\cG(\phi_1,a,b)$ is a frame.
    
    Since $h_1$ is a nonzero constant multiple of $\phi_1$, the frame property is the same for the two windows. Hence, $\cG(h_1,a,b)$ is a frame, which completes the proof.
\end{proof}

\renewcommand{\bibfont}{\footnotesize}
\bibliographystyle{plainnat}
\bibliography{bib}

\end{document}